\documentclass[11pt]{amsart}
\usepackage{amsmath, amssymb, amsthm, amscd, amsfonts}
\usepackage[colorlinks=true, linkcolor=blue,
urlcolor=blue]{hyperref}

\makeatletter \oddsidemargin.9375in \evensidemargin \oddsidemargin
\newtheorem{theorem}{Theorem}[section]
\newtheorem{lemma}[theorem]{Lemma}
\newtheorem{proposition}[theorem]{Proposition}
\newtheorem{corollary}[theorem]{Corollary}
\theoremstyle{definition}
 \newtheorem{definition}[theorem]{Definition}
\newtheorem{example}[theorem]{Example}

\newtheorem{remark}[theorem]{Remark}
\theoremstyle{remark}
\numberwithin{equation}{section}

\begin{document}

\title[Quasidiagonal traces]{ Quasidiagonal  traces and crossed products}


\author[Marzieh Forough ]{Marzieh Forough }
\address{ School of Mathematics, Institute for Research in Fundamental Sciences (IPM), P.O. Box 19395-5746, Tehran, Iran}
\email{mforough@ipm.ir, mforough86@gmail.com}


\begin{abstract}
Let $A$ be a simple, exact, separable, unital $C^*$-algebra and let $\alpha \colon G \rightarrow Aut(A)$ be an action of a finite group $G$ with the weak tracial Rokhlin property. We show that every trace on $A \rtimes_{\alpha} G$ is quasidiagonal provided that all traces on $A$ are  quasidiagonal. As an application, we study the behavior of finite decomposition rank under taking crossed products by finite group actions with the weak tracial Rokhlin property. Moreover, we discuss the stability of the property that all traces are quasidiagonal under taking crossed products of finite group actions with finite Rokhlin dimension with commuting towers.
\end{abstract}

\maketitle
\tableofcontents

\section{Introduction}
In this paper, we study the stability of the property that all traces on a $C^*$-algebra are quasidiagonal by  taking crossed product by a finite group action with some form of the Rokhlin property. In particular, we consider crossed products of finite group actions with the weak tracial Rokhlin property and with finite Rokhlin dimension with commuting towers.

The notion of quasidiagonality appeared in Halmos' work on operator theory. Then it was introduced in the theory of $C^*$-algebras: a $C^*$-algebra is quasidiagonal if it admits a faithful representation for which there is an approximately central net of finite rank projections converging strongly to the unit. Voiculescu proved that for a separable unital $C^*$-algebra $A$, quasidiagonality is equivalent to the existence of a sequence of unital completely positive maps from $A$ to a matrix algebra which are both asymptotically multiplicative and asymptotically isometric. In the same spirit, Brown introduced quasidiagonal traces and investigated their properties in \cite{Brown}. 
Recently, the property that all traces on a $C^*$-algebra are quasidiagonal has been investigated deeply regarding the classification program of $C^*$-algebras. In \cite{Brown-Sato}, the authors argued that the condition that all traces are quasidiagonal distinguishes decomposition rank from nuclear dimension. 
 Moreover, in \cite{Brown-Sato}, it was shown that all traces on simple, separable, unital $C^*$-algebras with finite decomposition rank are quasidiagonal. Tikuisis, Winter and White in \cite{Winter-White} have proved that every trace on a simple, separable, nuclear  $C^*$-algebra in the $UCT$ class is quasidiagonal. This result has significant consequences in the Elliott classification program, see \cite{Winter-White}. In other direction, it was proved by Rosenberg and Tikuisis-Winter-White that a discrete group is amenable if and only if the canonical trace on its reduced group $C^*$-algebra is quasidiagonal.  It has been shown by Elliott-Gong-Lin-Niu in \cite{EGLN} that quasidiagonality of all traces is fundamental in the classification of simple, nuclear, stably finite $C^*$-algebras.  
 
 These results suggest that it is important to determine when all traces are quasidiagonal.  In particular, we investigate when all traces  on a crossed product algebra by a finite group action are quasidiagonal. It has been proved that several properties appearing in Elliott's classification are stable by taking crossed products by finite group actions with some kind of Rokhlin property. This motivates us to investigate when all traces on a crossed product algebra by a finite group action with some type of Rokhlin property are quasidiagonal. 

The Rokhlin property for actions on $C^*$-algebras  was investigated by Kishimoto, Herman and Jones, Herman and Ocneanu. Izumi in \cite{Izu04} gave a modern definition of the Rokhlin property for finite group actions acting on unital  $C^*$-algebras and classified finite group actions on some classes of unital  $C^*$-algebras with the Rokhlin property.
 It has been studied extensively  which properties of $C^*$-algebras are preserved by taking crossed products of  group actions with the Rokhlin property, see \cite{Phillips-tracial}, \cite{Osaka-P}, \cite{Hirshberg}, \cite{Zacharias}, and \cite{Gardella}. 
 The Rokhlin property is quite restrictive. There are relatively few actions with the Rokhlin property and there are many algebras which admit no finite group action with the Rokhlin property. Phillips in \cite{Phillips-tracial}  introduced the tracial Rokhlin property for finite group actions that is less restrictive version of the Rokhlin property. There are many examples of actions of finite groups with the tracial Rokhlin property, see \cite{Phillips}.  Phillips in \cite{Phillips-tracial} proved that having tracial rank zero is stable under taking crossed products by finite group actions with the tracial Rokhlin property. 
 
  The tracial Rokhlin property still needs the existence of projections, which can be restrictive. For instance, the Jiang-Su algebra $\mathcal{Z}$ does not admit any finite group actions with the tracial Rokhlin property. Projection-free generalizations of the tracial Rokhlin property were considered in \cite{Sato-M}, \cite{Hirshberg-1}, and \cite{Gardella2}, among others.  In particular, preservation of (tracial) $\mathcal{Z}$-stability by taking crossed products by finite group actions was investigated in \cite{Hirshberg-1}.
  

Another approach to weakening the  Rokhlin property is given in \cite{Winter-Z}. The authors define and study Rokhlin dimension for finite group actions. The preservation of nuclear dimension, decomposition rank and $\mathcal{Z}$-stability is investigated by taking crossed products of finite group actions with finite Rokhlin dimension in \cite{Zacharias}. Continuation of this research line and motivated by the recent advances in the classification program of  $C^*$-algebra, we study the stability of the property that all traces are quasidiagonal by taking crossed products of finite group actions. The following is our main theorem. 
\begin{theorem}\label{main1}
Let $A$ be a simple, separable, exact,  unital $C^*$-algebra and let $\alpha \colon G \rightarrow Aut(A)$ be an action of a finite group $G$ with the weak tracial Rokhlin property.  If every trace on $A$ is quasidiagonal then all traces on $A \rtimes _{\alpha} G$ are quasidiagonal.
\end{theorem}
Gardella, Hirshberg and Santiago in \cite{Gardella2} have studied the relation between the weak tracial Rokhlin property and having finite Rokhlin dimension with commuting towers. Their result together with Theorem \ref{main1} enables us to conclude some conditions under which all traces on a crossed product by a finite group action with finite Rokhlin dimension with commuting towers are quasidiagonal, see Corollary \ref{commuting towers}.

 As an application of Theorem \ref{main1}, we study finiteness of the decomposition rank of  crossed products by finite group actions with the tracial Rokhlin property. 
\begin{corollary} 
Let $A$ be a simple, separable, unital $C^*$-algebra with unique trace and let $\alpha \colon G \rightarrow Aut(A)$ be an action of a finite group $G$ with the tracial Rokhlin property. If  the decomposition rank of $A$ is finite, then the decomposition rank of $A \rtimes _{\alpha} G$ is finite.
\end{corollary}
We also show that the decomposition rank of the crossed product can be different from the decomposition rank of the original algebra, see  Example ~\ref{ex.dr}.\\

This paper is organized as follows. In section 2, we study $C^*$-algebras all of whose traces are (uniform) quasidiagonal.  In section 3, we study  quasidiagonality of all traces on a crossed product by a finite group action with some type of Rokhlin property when all traces on the original algebra are quasidiagonal. We mainly consider finite group actions with the weak tracial Rokhlin property or with finite Rokhlin dimension with commuting tower.
\section{Quasidiagonal traces}
We begin this section by recalling the definition of quasidiagonal traces and uniform quasidiagonal traces from \cite{Brown}. The canonical trace on $M_{k}$ will be denoted by $tr_{k}$.
\begin{definition} \label{qd1}
Let $A$ be a $C^*$-algebra and $\tau$ be a trace on $A$.  We say that  $\tau $ is quasidiagonal if there exists a net of  completely positive contractive (c.p.c)  maps $\phi_{i} \colon A \rightarrow M_{k(i)}$  such that 
\begin{itemize}
\item[(1)]
$\|tr_{k(i)}\circ \phi_{i}(a) -  \tau (a) \| \rightarrow 0$;
\item[(2)]
 $\|\phi_{i}(ab)-\phi_{i}(a)\phi_{i}(b)\| \rightarrow 0$,
 \end{itemize}
  for all $a, b \in A$. 
  \end{definition}
  \begin{definition}\label{uqd}
  Let $A$ be $C^*$-algebra and $\tau$ be a trace on $A$.  We say that  $\tau $ is uniform quasidiagonal if there exists a net of  completely positive contractive (c.p.c)  maps $\phi_{i} \colon A \rightarrow M_{k(i)}$  such that 
  \begin{itemize}
\item[(1)]
$\|tr_{k(i)}\circ \phi_{i} -  \tau \|_{A^{*}} \rightarrow 0$;
\item[(2)]
 $\|\phi_{i}(ab)-\phi_{i}(a)\phi_{i}(b)\| \rightarrow 0$,
 \end{itemize} 
  for all $a, b \in A$. 
 \end{definition}
 
 \begin{remark}
 By Proposition 3.5.10 of \cite{Brown}, if $A$ is a unital $C^*$-algebra, we can take the  $\phi_{i}$'s in Definition ~\ref{qd1} and Definition ~\ref{uqd} to be unital completely positive (u.c.p) maps.
 \end{remark}
 Brown in \cite [Theorem~1]{Brown} proves that if $A$ is a locally reflexive  $C^*$-algebra then every  quasidiagonal trace on $A$ is uniform quasidiagonal. In particular, any quasidiagonal trace on an exact $C^*$-algebra is uniform quasidiagonal. 
 
 In the following remark, we mention some classes of $C^*$-algebras all of whose traces are quasidiagonal.  
\begin{remark}
\begin{itemize}
\item Let $A$ be a nuclear quasidiagonal $C^*$-algebra with unique trace $\tau$, then $\tau$ is  quasidiagonal (Theorem 6.1.13 of \cite{Brown}).
\item Let $A$ be a separable unital $C^*$-algebra with finite decomposition rank, then every trace on $A$ is quasidiagonal (Corollary 8.7 of \cite{Brown-Sato}).
\item Every trace on a separable, simple, nuclear, quasidiagonal, unital $C^*$-algebra satisfying UCT is quasidiagonal (Corollary 6.1 of \cite{Winter-White}).
\end{itemize}
\end{remark}

We recall that a trace $\tau$ on a $C^*$-algebra $A$ is called amenable if there exist a net of c.p.c  maps $\phi_{\alpha} \colon A \rightarrow M_{k(\alpha)}$  such that $\|tr_{k(\alpha)}\circ \phi_{\alpha}(a) -  \tau (a) \| \rightarrow 0$, for all $ a \in A$ and
 $\|\phi_{\alpha}(ab)-\phi_{\alpha}(a)\phi_{\alpha}(b)\| _{2} \rightarrow 0$ for all $a, b \in A$. Clearly, every quasidiagonal trace is amenable but the converse does not necessarily hold, see \cite{Brown}.
\begin{lemma}\label{cone+amn}
Suppose that all traces on $A$ are amenable. Then every trace on $C_{0}(0,1] \otimes A$ is quasidiagonal.
\end{lemma}
\begin{proof}
Observe that every trace of the form $\delta _{t} \otimes \tau$ is amenable on $C_{0}((0,1])\otimes A$, where $\delta_{t}$ is the evaluation map for some point  $t \in (0,1]$ and $\tau$ is an amenable trace on $A$. Moreover, by Proposition 3.5.1 of \cite{Brown}, the set of all amenable traces on $C_{0}((0,1])\otimes A$ is a weak $*$-closed convex set. On the other hand, it is well-known fact that 
that every trace on $ C_{0}((0,1])\otimes A$ lies in the weak $*$-closed convex hull of the set of traces in the form $\delta _{t} \otimes \tau$ for some $t \in (0,1]$ and trace $\tau$ on $B$. Thus all traces on $C_{0}((0,1])\otimes A$ are amenable. By Proposition 3.2 of \cite{Brown-March}, every amenable trace on a cone of a $C^*$-algebra is quasidiagonal. Therefore, all traces on the cone of $A$ are quasidiagonal, as desired.
\end{proof}
Recall that a $C^*$-algebra $A$ is said to have the weak expectation property (WEP) if there exists a u.c.p map $\phi \colon B(H_{u}) \rightarrow A^{**}$ such that $\phi(a)=a$ for all $a \in A$, where $A \subseteq A^{**} \subseteq B(H_{u})$ is the universal representation of $A$. The class of $C^*$-algebras with $WEP$ is large and contains of injective $C^*$-algebras and nuclear $C^*$-algebras.
\begin{corollary}
Let $A$ be a $C^*$-algebra with the $WEP$. Then all traces on $C_{0}(0,1] \otimes A$ are quasidiagonal.
\end{corollary}
\begin{proof}
By Proposition 4.2.2 of \cite{Brown},  all traces on $A$ are amenable. Now, Lemma ~\ref{cone+amn} implies the result.
\end{proof}

 In the following, we aim to show that if all traces on a unital simple $C^*$-algebra $A$ are uniform quasidiagonal then any trace on a corner of $A$ is uniform quasidiagonal. To this end, we first need the following two results.
 
 \begin{lemma} \label{Simple}
Let $A$ be a simple unital $C^*$-algebra. If $p$ is a non-zero projection in $A$, then every trace on a hereditary $C^*$-subalgebra $pAp$ is of the form $\frac{1}{\tau(p)} \tau|_{pAp}$ for some trace $\tau$ on $A$.
\end{lemma}
\begin{proof}
Let $p$ be an arbitrary non-zero projection in $A$. Simplicity of $A$ implies that  $pAp$ is $A$ is Morita equivalent to $pAp$. Hence, it is easy to observe that $Ap$ is an imprimitivity $A$-$pAp$ bimodule. Now, Proposition 2.2 of  \cite{Rieffel} implies that all non-normalized traces on $pAp$ are of the form $ \tau (<. , .>_{pAp})$ for some trace $\tau$ on $A$. This completes the proof.
\end{proof}
 
 \begin{proposition} (due to N. Brown)  \label{Brown}
Let $\tau$ be a uniform quasidiagonal trace on a $C^*$-algebra $A$ and $p$ be a projection of $A$ such that $\tau(p)$ is non zero, then $\frac{1}{\tau(p)}\tau$ restricts to a uniform quasidiagonal trace on $pAp$.
\end{proposition}
\begin{proof}
Suppose that $(\phi_{\alpha} \colon A \rightarrow M_{k(\alpha)})_{\alpha \in \mathcal{I}}$ is a net of  c.c.p maps realizing the uniform quasidiagonality of $\tau$. Since $\phi_{\alpha}$ are asymptotically multiplicative, the positive contractions $\phi_{\alpha}(p)$ satisfy $\|\phi_{\alpha}(p)-\phi_{\alpha}(p)^{2}\| \rightarrow 0$. By functional calculus we can therefore find projections $P_{\alpha} \in M_{k(\alpha)}(\mathbb{C})$ such that $\|\phi_{\alpha}(p)-P_{\alpha}\| \rightarrow 0$. We claim that the maps which witness that $\frac{1}{\tau(p)}\tau$ restricts to a uniformly quasidiagonal trace on $pAp$ are given by $\varphi_{\alpha}(pap)=P_{\alpha}\phi_{\alpha}(pap)P_{\alpha}$.  In the other words, we will show 

$(1)$ $(\varphi_{\alpha})_{\alpha \in \mathcal{I}}$ are asymptotically multiplicative;

$(2)$ and for every $\varepsilon >  0$, there exists $\alpha_{0}$ such that $|\frac{1}{\tau(p)}\tau(pxp)-\frac{1}{tr(P_{\alpha})} tr(\varphi_{\alpha}(pxp))| \leq \epsilon$ for all contractions $x \in A$ and $\alpha > \alpha_{0}$. 

To prove both of these assertions, we require the following claim.
 
 Claim: For every $\varepsilon > 0$, there exists $\alpha_{0}$ such that $\|P_{\alpha}\phi_{\alpha}(pxp)-\phi_{\alpha}(pxp)\| \leq \epsilon$ for any $\alpha > \alpha_{0}$. 
 
To prove the claim, we first note that by Lemma 3.5 of \cite{Kirchberg}, we have 
$$\|\phi_{\alpha}(px)-\phi_{\alpha}(p)\phi_{\alpha}(px)\| \leq \|\phi_{\alpha}(p)-\phi_{\alpha}(p)^{2}\|^{\frac{1}{2}}$$
for all contractions $x \in A$. Since  $\|\phi_{\alpha}(p)-\phi_{\alpha}(p)^{2}\|\rightarrow 0$ and $\|\phi_{\alpha}(p)-P_{\alpha}\| \rightarrow 0$, it follows that for every $\epsilon > 0$, there exists $\alpha_{0}$ such that $\|\phi_{\alpha}(px)-P_{\alpha}\phi_{\alpha}(px)\| \leq \epsilon $ for all contractions $x \in A$ and $\alpha > \alpha_{0}$. Taking adjoints we get the same inequalities with $p$ and $P_{\alpha}$ on the right side of $x$, and so some standard estimates complete the proof of claim.

With above claim we verify $(1)$:
$$\varphi_{\alpha}(pxppyp)=P_{\alpha}\phi_{\alpha}(pxppyp)P_{\alpha} \approx \phi_{\alpha}(pxppyp) \approx \phi_{\alpha}(pxp)\phi_{\alpha}(pyp) \approx \varphi_{\alpha}(pxp)\varphi_{\alpha}(pyp).$$
To verify $(2)$ one observes that 
$$|\frac{1}{\tau(p)}\tau(pxp)-\frac{1}{tr(P_{\alpha})} tr(\varphi_{\alpha}(pxp))|$$
is bounded above by the sum of
$$|\frac{1}{\tau(p)}\tau(pxp)-\frac{1}{tr(P_{\alpha})} tr(\phi_{\alpha}(pxp))|$$
and
$$|\frac{1}{tr(P_{\alpha})} tr(\phi_{\alpha}(pxp))-\frac{1}{tr(P_{\alpha})} tr(\varphi_{\alpha}(pxp))|.$$
Now, it is easy to see $(2)$, and this completes the proof.
\end{proof}

\begin{corollary}
Let $A$ be a unital simple $C^*$-algebra whose all traces are uniform quasidiagonal. If $p$ is any projection on $A$, then all of whose traces on $pAp$ are uniform quasidiagonal.
\end{corollary}
 \begin{proof}
 This follows immediately from Proposition~\ref{Brown} and Lemma ~\ref{Simple}.
 \end{proof}
 \section{Quasidiagonal traces on crossed product $C^*$-algebras}
 In this section, we investigate the behavior of quasidiagonal traces by taking crossed products of finite group actions with the weak tracial Rokhlin property and with finite Rokhlin dimension with commuting towers. 
 
 We begin by recalling the definition of the tracial Rokhlin property for finite group actions and the weak version of it.
 \begin{definition} \cite{Phillips-tracial}
  Let $A$ be an infinite dimensional simple separable unital $C^*$- algebra, and let $\alpha \colon G \rightarrow Aut(A)$ be an action of a finite group $G$ on $A$. We say that $\alpha$ has the tracial Rokhlin property if for every finite set $F \subseteq A$, every $\varepsilon > 0$, and every positive element $x \in A$ with $\|x\| = 1$, there are mutually orthogonal projections $e_{g} \in A$ for $g \in G $ such that
  \begin{itemize}
  \item[(1)]
 $\|e_{g}a-ae_{g}\| \leq \varepsilon$ for all $a \in F$ and $g \in G$;
 \item[(2)]
   $\|\alpha_{g}(e_{h})-e_{gh}\| \leq \varepsilon$ for all $g, h \in G$;
   \item[(3)]
 With $e = \sum_{g \in G} e_{g}$, the projection $1-e$ is Murray-von Neumann equivalent to a projection in the hereditary subalgebra of $A$ generated by $x$;
 \item[(4)]
With $e$ as in $(3)$, we have $\|exe\| \geq 1-\varepsilon$.
\end{itemize}
 \end{definition}  
 For  positive elements $a, b$ of $C^*$-algebra $A$, we write $a\precsim b$ if $a$ is Cuntz subequivalent to $b$, i.e., there is a sequence $(v_n)$ in $A$ such that $\|a-v_{n}bv_{n}^{*}\|\to 0$. We write $a\sim b$ if both $a\precsim b$ and $b\precsim a$.
  \begin{definition} \cite{Gardella2}
Let $\alpha \colon G \rightarrow Aut(A)$ be finite group action on a simple separable unital $C^*$-algebra $A$. We say that $\alpha$ has the weak tracial Rokhlin  property if for every $\varepsilon >0$, for every finite set $F \subseteq A$ and for every positive  $x \in A$ with norm one, there exist orthogonal contractions $f_{g} \in A$, for all $ g \in G$, satisfying
\begin{itemize}
\item[(1)]
$ \|\alpha_{g}(f_{h})-f_{gh}\| \leq \varepsilon$;
\item[(2)]
 $\|f_{g}a-af_{g}\| \leq \varepsilon$;
 \item[(3)]
With $f=\sum_{g \in G} f_{g}$, $1-f \precsim x$;
\item[(4)]
 With $f$ as in $(3)$, we have $\|fxf\| >1-\varepsilon$.
 \end{itemize}
\end{definition}

 \begin{proposition}\label{Tehran} (due to E. Gardella)
Let $\alpha \colon G \rightarrow Aut(A)$ be an action of a finite group $G$ on a simple separable unital $C^*$-algebra $A$ and let $\omega \in \beta \mathbb{N} / \mathbb{N}$. If $\alpha$ has the weak tracial Rokhlin property, then there exists an equivariantly c.p.c order zero map $\phi \colon C(G) \rightarrow A_{\omega} \cap A'$ such that $1-\phi(1_{C(G)}) \in J_{A}$, where $J_{A}= \{a \in A_{\omega},  lim _{n \rightarrow \omega} sup _{\tau \in T(A)} \tau_{\omega}(a^{*}a)=0 \}$.
\end{proposition}
\begin{proof}
Let $x$ be an element in $J_{A}$ with $\|x\| \leq 1$ and let the sequence $\{x_{n}\} $ of positive contractions in $A$ be a lift for $x$.  Suppose that $\{F_{n}\}$ is a sequence of finite sets in $A$ such that $\cup _{n \in \mathbb{N}}F_{n}$ is dense in $A$. There exist positive contractions $f_{g}^{(n)}$, for $g \in G$, satisfying conditions for the weak tracial Rokhlin property for finite set $F_{n}$, $\varepsilon =\frac{1}{n}$ and $x_{n}$.  Denote the $\sum_{g \in G} f_{g}^{(n)}$ by $f^{(n)}$. Then $d_{\tau}(1-f^{(n)}) \leq d_{\tau}(x_{n})$ for any  $\tau \in T(A)$, where $d_{\tau}(a)=lim _{n \rightarrow \infty} \tau(a^{1/n})$ for all $a \in A$. With $f_{g}=(f_{g}^{(n)})_{n \in \mathbb{N}}$, it is easy to see that $f_{g}$ belongs to $A_{\omega} \cap A^{'}$.   With identifying $ C(G)$ with $\mathbb{C}G$, put $\phi(g)= f_{g}$. This defines the desired map $\phi$.
\end{proof}
\begin{remark}
The converse of Proposition \ref{Tehran} holds if we moreover assume that $A$ has the strict comparison property.
\end{remark}

\begin{theorem} \label{w.t.r.p+qd}
  Let $A$ be a simple, separable, exact, unital $C^*$-algebra, and let $\alpha \colon G \rightarrow Aut(A)$ be a finite group action with the weak tracial Rokhlin property. Then every trace on $A \rtimes _{\alpha} G$ is quasidiagonal if all traces on $A$ are quasidiagonal.
  \end{theorem}
   \begin{proof}
  It follows from Proposition ~\ref{Tehran} that there is a c.p.c order zero map $\phi \colon C(G) \rightarrow A_{\omega} \cap A'$ such that $1-\phi(1_{C(G)}) \in J_{A}$. Hence, one can find matually orthogonal positive contractions $(f_{g})_{g \in G}$ in $A_{\omega} \cap A'$ such that $\phi(1_{C(G)})= \sum _{g \in G} f_{g}$ and $\alpha_{g}(f_{e})=f_{g}$ for all $g \in G$. Then define  $\psi \colon A \rightarrow (A ^{\alpha})_{\omega}$ by $\psi(a)=  \sum _{ \in G}  \alpha_{g}(f_{e}^{1/2}a f_{e} ^{1/2})$. Clearly, $\psi$ is a c.p.c map. Note that $\psi(a)= \sum_{g \in G} f_{g} ^{1/2}\alpha_{g}(a) f_{g}^{1/2}=  \sum_{g \in G} f_{g} \alpha_{g}(a)=\sum_{g \in G} \alpha_{g}(a) f_{g}$ for any $a \in A$. Thus $\psi$ is order zero since $f_{g}f_{h}=0$ when $g \neq h$ and $f_{g} \in A_{\omega} \cap A'$.  Moreover, for every $a \in A ^{\alpha}$, we have $\psi(a) -a= \sum_{g \in G} f_{g} \alpha_{g}(a)- a= a(\sum_{g \in G} f_{g}- 1)= a(\phi(1_{C(G)}-1) \in J_{A}$ as $J_{A}$ is an ideal.
  
 Now, we show that all traces on $A^{\alpha}$ are quasidiagonal. Let $\tau$ be a trace on $A^{\alpha}$, then it induces a trace $\tau_{\omega}$ on $(A ^{\alpha})_{\omega}$. By Corollary 4.4 of \cite{Winter-Z}, $\tau _{\omega} \circ \psi$ is a trace on $A$ and so it is quasidiagonal. Suppose the finite set $F \subseteq A^{\alpha}$ and $\varepsilon > 0$ are given. Then there is a c.p.c map $\phi \colon A \rightarrow M_{n}$ such that $\|\phi(ab) - \phi(a)\phi(b)\| \leq \varepsilon$ and $\|tr_{n} (\phi(a))-\tau _{\omega}( \psi(a))\| \leq \varepsilon$ for all $a, b \in F$. Note that $\tau _{\omega}( \psi(a))=\tau(a)$ for all $a \in A$ since $\psi(a) -a \in J_{A}$. Thus, the restriction of $\psi$ on $A ^{\alpha}$ is almost multiplicative on $F$ and $\|tr_{n} (\phi(a))-\tau(a))\| \leq \varepsilon$ for all $a, b \in F$. Therefore, we proved that $\tau$ is quasidiagonal. Indeed, all traces on $A^{\alpha}$ are uniform quasidiagonal, since $A ^{\alpha} \subseteq A$ is exact.
   Now, we prove that every trace on $A\rtimes _{\alpha}G$ is quasidiagonal. By Proposition 5.3 of \cite{Hirshberg-1}, $\alpha$ is outer, and so $A \rtimes _{\alpha}G$ is simple. Thus $A^{\alpha}$ is Morita equivalent to $A \rtimes _{\alpha} G$. Since both algebras are separable and unital, there are $n \in \mathbb{N}$ and projection $p \in M_{n} \otimes A^{\alpha}$ such that $A \rtimes _{\alpha}G \cong p(M_{n} \otimes A^{\alpha})p$. Therefore, it follows from Lemma ~\ref{f.sat} that all traces on $A \rtimes _{\alpha} G$ are uniform quasidiagonal. Since $A \rtimes _{\alpha}G$ is exact, all traces on $A \rtimes _{\alpha}G$ are quasidiagonal, as desired.
\end{proof}

 \begin{corollary} \label{decomp+rank}
  Let $A$ be a simple, separable, unital $C^*$-algebra with unique trace,  and let $\alpha \colon G \rightarrow Aut(A)$ be a finite group action with the  tracial Rokhlin property. Suppose $A$ has finite decomposition rank, then the decomposition rank of $A\rtimes _{\alpha}G$ is at most one.
  \end{corollary}
  \begin{proof}
  First note that by Proposition 5.7 of \cite{Echterhoff-Phillips}, any trace on $A\rtimes _{\alpha}G$ is the restriction of a $\alpha$-invariant trace on $A$. Moreover, the unique trace of $A$ is $\alpha$-invariant and so $A \rtimes _{\alpha} G$ has a unique trace. Since $A$ has finite decomposition rank and the trace space of $A$ is a Bauer simplex, Corollary 8.6 of \cite{Brown-Sato} implies that $A$ is nuclear, $\mathcal{Z}$-stable and all traces on $A$ are quasidiagonal.  Note that Corollary 5.7 of \cite{Hirshberg} implies that $A \rtimes _{\alpha}G$ is $\mathcal{Z}$-stable. Moreover, it follows from Theorem  ~\ref{w.t.r.p+qd} that all traces on $A \rtimes _{\alpha}G$ are quasidiagonal.  Since $A \rtimes _{\alpha}G$ has a unique trace, we can conclude from Corollary 8.6 of \cite{Brown-Sato} that the decomposition rank of $A \rtimes _{\alpha}G$ is at most one, as desired.
  \end{proof}
  
  We remark here that unique trace property in Corollary ~\ref{decomp+rank} is not necessary. Indeed, it is enough to assume that the trace space of $A$ is Bauer simplex.
 
  In the following example, we show that the decomposition rank of a crossed product of a simple $C^*$-algebra with the unique trace by a finite group action with the tracial Rokhlin property does not necessarily equal to the decomposition rank of the original algebra. 
 \begin{example}\label{ex.dr}
One can construct an example of a finite group action $\alpha \colon G \rightarrow Aut(A)$ satisfying

$(1)$ $\alpha$ has the tracial Rokhlin property,

$(2)$ $\alpha$ has infinite Rokhlin dimension with commutative towers,

$(3)$ $A$ is a unital, separable, $\mathcal{Z}$-stable $C^*$-algebra such that both $A$ and $A \rtimes_{\alpha}G$ have unique trace,\\
such that $dr(A\rtimes_{\alpha}G)$ is not equal to $dr(A)$.
 
 In \cite{Black}, Blackadar constructed an example of a $\mathbb{Z}_{2}$-action on the $UHF$-algebra $A$ of type $2^{\infty}$, whose crossed product is not $AF$.  Phillips in Proposition 3.4 of \cite{Phillips} showed that $\alpha$ has the tracial Rokhlin. Note that $A$ has a unique trace $\tau$ and so $\tau$ is $G$-invariant. Now, employ Proposition 5.7 of \cite{Echterhoff-Phillips} to conclude that $A\rtimes _{\alpha}G$ has a unique trace. Thus the trace spaces of $A$ and $A \rtimes_{\alpha}G$ are Bauer simplex.  However, from Example 2.9 of \cite{Gardella2}, $\alpha$ has infinite Rokhlin dimension with commutating towers.
 Since $A \rtimes  _{\alpha} \mathbb{Z}_{2}$ is not $AF$, the action $\alpha$ does not have the Rokhlin property. Note that $A$ is $\mathcal{Z}$-stable with unique trace which is quasidiagonal. Thus by Corollary \ref{decomp+rank}, $dr (A \rtimes  _{\alpha} \mathbb{Z}_{2}) \leq 1$. Since $A \rtimes  _{\alpha} \mathbb{Z}_{2}$ is not an $AF$-algebra, $dr (A \rtimes  _{\alpha} \mathbb{Z}_{2}) =1$. Finally, note that $dr(A)=0$ since $A$ is an $AF$-algebra.
  \end{example} 
  
  We recall the definition of Rokhlin dimension from \cite{Zacharias}.
  
  \begin{definition}
Let $G$ be a finite group, let $A$ be a separable unital $C^*$-algebra, and let $\alpha \colon G \rightarrow Aut(A)$ be an action of $G$ on $A$. Given a non negative integer $d$, we say that $\alpha$ has Rokhlin dimension $d$, we denote by $dim _{Rok} (\alpha)=d$, if $d$ is the least integer with the following property: For any finite set $F \subseteq A$, $\varepsilon >0$ there exist positive contractions $f_{g}^{(l)}$; $g \in G$; $l=0,...,d$ satisfying the following condition for every $l, k=0,...,d$, for every $g, h \in G$ and $a \in F$
 
\begin{itemize}
\item[(1)]$ \|\alpha_{h}(f_{g}^{(l)})- f_{hg}^{(l)}\| \leq \varepsilon$;
\item[(2)] $\|f_{g}^{(l)}a-af_{g}^{(l)}\| \leq \varepsilon$;
\item[(3)] $\|f_{g}^{(l)}f_{h}^{(k)}\| \leq \varepsilon$ when $k \neq l$;
\item[(3)] $\|(\sum_{l} \sum _{g \in G} f_{g}^{(l)})-1\| \leq \varepsilon$.
\end{itemize}
\end{definition}
If one can always choose the positive contractions $f_{g}^{(l)}$ above to moreover satisfy  $\|[f_{g}^{(l)}, f_{h}^{(k)}\| \leq \varepsilon$ for all $h, g \in G$ and $l, k= 0,..., d$, we say that $\alpha$ has Rokhlin dimension with commuting towers, and denote $dim _{Rok} ^{c} (\alpha)=d$.
  
 The relation between the tracial Rokhlin property and finite  Rokhlin dimension with commuting towers is discussed in Theorem 2.3 of \cite{Gardella2}.
 
 \begin{theorem}\cite[Theorem ~3.2] {Gardella2} \label{Gard}
  Let A be an infinite dimensional, simple, finite, unital $C^*$-algebra with strict comparison and at most countably many extreme tracial states, and let $\alpha \colon  G \rightarrow Aut(A)$ be a finite group action. If $dim_{Rok}^{c}(\alpha) < \infty$, then $\alpha$ has the weak tracial Rokhlin property.
  \end{theorem}
  Theorem ~\ref{Gard}  together with Theorem \ref{w.t.r.p+qd} enable us to obtain the following.
  \begin{corollary}\label{commuting towers}
  Let $A$ be a simple, exact, finite, separable unital $C^*$-algebra with strict comparison and at most countably many extreme traces. Let $\alpha \colon G \rightarrow Aut(A)$ be an action of a finite group with finite Rokhlin dimension with commuting towers. Then every trace on $A \rtimes_{\alpha}G$ is quasidiagonal if all traces on $A$ are quasidiagonal.
  \end{corollary}
  
  However, in the case of crossed products by actions with the tracial Rokhlin property, we can study the behavior of quasidiagonality of traces on crossed product when the original $C^*$-algebra is not necessarily exact. For this, we use the following notion. 

\begin{definition}\label{TR} \cite{Elliott}
Let $\mathcal{C}$ be a class of separable unital $C^*$-algebras. The class of unital $C^*$-algebras which are  tracially approximated by $C^*$-algebras in $\mathcal{C}$, denoted by $TA\mathcal{C}$, is defined as follows.  A unital $C^*$-algebra $A$ is said to belong to the class $TA \mathcal{C}$ if for any $\varepsilon  > 0$, any finite subset $F \subseteq  A$,  and any non-zero $a \in A_{+}$ with $\|a\|=1$, there exist a non-zero projection $p \in A$ and a $C^{*}$-subalgebra $C \subseteq A$  with unit $p$ such that $C \in \mathcal{C}$, and for all $x \in F$,\\
$(i)$ $\|xp-px\| \leq \varepsilon$,\\
$(ii)$ $pxp \subseteq _{\varepsilon} C$,\\
$(iii)$ $1-p$ is Murray-von Neumann equivalent to a projection in $\overline{aAa}$. 
\end{definition}

We recall the following result form \cite{Osaka3}.
 
\begin{proposition} \label{TR} \cite[Theorem 3.3]{Osaka3}
Let $\mathcal{C}$ be a class of separable unital $C^*$-algebras such that 
\begin{itemize}
\item[(1)]
 If $A\in \mathcal{C}$ and $B \simeq A$, then $B \in \mathcal{C}$;
 \item[(2)]
 If $A \in \mathcal{C}$ and $n$ is any integer, then $M_{n}(A) \in \mathcal{C}$;
 \item[(3)]
 If $A \in \mathcal{C}$ and $p \in A$ is a nonzero projection, then $pAp \in \mathcal{C}$.
 \end{itemize}
 Suppose that $\alpha$ is an action of finite group $G$ on a simple separable unital $C^*$-algebra  $A$ with the tracial Rokhlin property. If $A$ is a $TA\mathcal{C}$-algebra then $A \rtimes _{\alpha}G$ is in $TA\mathcal{C}$.
\end{proposition}

Let us recall the next proposition from \cite{Brown-Sato}.
 
\begin{proposition} \cite[Proposition 8.3]{Brown-Sato}
Let denote by  $\mathcal{C}_{qd}$ the class of all separable unital $C^*$-algebras all of whose traces are quasidiagonal. If $A$ is a simple separable unital $C^*$-algebra in $TA \mathcal{C}_{qd}$ then $A \in \mathcal{C}_{qd}$.
\end{proposition} 

The following lemma can be deduced from Proposition ~\ref{Brown} and Proposition 3.7 of ~\cite{Brown} and Lemma ~\ref{Simple}.
\begin{lemma} \label{f.sat}
The class of all simple unital $C^*$-algebras in $\mathcal{C}_{u.qd}$ satisfies in the following conditions:
 \begin{itemize}
\item[(1)]
 If $A\in \mathcal{C}$ and $B \simeq A$, then $B \in \mathcal{C}$;
 \item[(2)]
 If $A \in \mathcal{C}$ and $n$ is any integer, then $M_{n}(A) \in \mathcal{C}$;
 \item[(3)]
 If $A \in \mathcal{C}$ and $p \in A$ is a nonzero projection, then $pAp \in \mathcal{C}$.
 \end{itemize}  
 \end{lemma}
\begin{corollary} \label{qd+tracial}
Let $A$ be a simple separable unital $C^*$-algebra all of whose traces are uniform quasidiagonal and let $\alpha$ be an action of a finite group $G$ on $A$ with the tracial Rokhlin property. Then all traces on $A \rtimes _{\alpha} G$ are quasidiagonal.
\end{corollary}
\begin{proof}
Let $\mathcal{C}_{u.qd}$ be a class of all separable unital $C^*$-algebras all of whose traces are uniform quasidiagonal, then by Lemma ~\ref{f.sat}, $\mathcal{C}_{u.qd}$ satisfies in conditions $(1)$ to $(3)$ of Proposition ~\ref{TR}, so $A \rtimes _{\alpha}G$ is a $TA\mathcal{C}_{u.qd}$-algebra. Now, employ Proposition ~ 8.3 of \cite{Brown-Sato} to deduce that all traces on $A \rtimes _{\alpha} G$ are quasidiagonal.
\end{proof}

{\textbf{Acknowledgments.}}
The author is grateful to Nate Brown for showing us the proof of Proposition \ref{Brown}. The author wishes to thank Eusebio Gardella for several useful comments and discussions. He suggested Proposition \ref{Tehran}.  She also would like to thank Stuart White,  Aaron Tikuisis and Naser Golestani for their comments. Finally, she is grateful to the referee for several comments and suggestions. Part of this work was obtained during a visit at the department of Mathematics at the University of Muenster. This research was supported by a grant from IPM and by the Deutsche Forchungsgemeinschaft (SFB 878).

\end{document}